\documentclass{amsart}
\pdfoutput=1

\usepackage{graphicx}

\newcommand{\parag}[1]{\medbreak{\em\bf#1}}

\def\sm{\operatorname{sm}}
\def\cm{\operatorname{cm}}
\def\sn{\operatorname{sn}}
\def\cn{\operatorname{cn}}
\def\dn{\operatorname{dn}}

\def\C{\mathbb{C}}
\def\Q{\mathbb{Q}}
\def\Z{\mathbb{Z}}

\def\K{\mathbb{K}}

\def\ds{\displaystyle}
\def\cal{\mathcal}
\def\frak{\mathfrak}
\def\Res{\operatorname{Res}}

\newcommand{\Img}[2]{\includegraphics[width=#1truecm]{#2}}

\newtheorem{definition}{Definition}
\newtheorem{theorem}{Theorem}
\newtheorem{proposition}{Proposition}
\newtheorem{corollary}{Corollary}

\begin{document}

\title[Pseudo-factorials, elliptic functions, and continued fractions]{Pseudo-factorials, elliptic 
functions, and continued fractions}
\author{Roland Bacher and Philippe Flajolet}
\date{January 12, 2009; revised May 31, 2009. To appear in  \emph{The Ramanujan Journal}.}
\address{R.B.: \rm Institut Fourier (CNRS UMR 5582), 100 rue des Maths, 
Universit\'e de Grenoble I, F--38402 Saint Martin d'H\`eres, France.}
\address{P.F.: \rm {\sc Algorithms}, INRIA-Rocquencourt, F--78150 Le Chesnay, France}

\begin{abstract}
This study presents miscellaneous properties
of pseudo-factorials, which are numbers whose recurrence
relation is a twisted form of that of usual factorials. 
These numbers are associated with special elliptic functions, 
 most notably, a Dixonian  and a Weierstra{\ss} 
function, which parametrize the Fermat cubic curve and are 
relative to a hexagonal lattice.
A 
continued fraction expansion of the ordinary generating function
of pseudo-factorials, first discovered empirically, is 
established here.  This article  also provides a characterization 
of the associated orthogonal polynomials, which appear to form a new family
of ``elliptic polynomials'', as well as  various other properties of pseudo-factorials, including 
a hexagonal lattice sum expression and elementary congruences.
\end{abstract}

\maketitle


\section{\bf Pseudo-factorials} 
Start from the innocuous looking recurrence, 
\begin{equation}\label{alphadef}
\alpha_{n+1}=(-1)^{n+1}
\sum_{k=0}^n \binom{n}{k}\alpha_k\alpha_{n-k}, \qquad \alpha_0=1,
\end{equation}
which determines a sequence of integers with initial values
\begin{equation}\label{inivals}
1,~-1,~-2,~2,~16,~-40,~-320,~1040,~12160,~-52480,~-742400\,.
\end{equation}
These numbers will be called \emph{pseudo-factorials}, since the omission
of the sign alternation $(-1)^{n+1}$ in Equation~\eqref{alphadef} determines
the sequence of factorials, $0!,1!,2!,\ldots$ 
At the suggestion of one of us (Bacher, October 2004), they have been included 
as Sequence {\bf A098777} in 
the \emph{On-line Encyclopedia of Integer Sequences}~\cite{Sloane08},
which is brilliantly maintained by Sloane and a gang of dedicated volunteers.
The purpose of this note is to show that these numbers, though not classical,
have a host of interesting properties.

The \emph{exponential generating function},
\begin{equation}\label{egf}
f(z):=\sum_{n\ge0} \alpha_n \frac{z^n}{n!}
=1-z-z^2+\frac{z^3}{3}+2\frac{z^4}{3}-\frac{z^5}{3}-\cdots,
\end{equation}
 is fundamental to our  treatment.  (The fact that the  
absolute values $|\alpha_n|$
are domi\-nated  by factorials implies that $f(z)$ is analytic
at least in $|z|<1$.)
We first elucidate the
relation between~$f(z)$ and  elliptic functions
of a kind introduced by Alfred Cardew Dixon in~1890 (\emph{vide}~\cite{Dixon90}),
then  show the reduction to the more common Weierstra{\ss} framework of $\wp$--functions: this 
forms the subject of Sections~\ref{dixon-sec} and~\ref{weier-sec}.
A simple consequence of the elliptic connections, worked out in Section~\ref{lat-sec}, 
is an expression of pseudo-factorials as 
sums over a hexagonal lattice.
Next,
we establish a continued fraction of a new type,
relative to the \emph{ordinary generating function} of
pseudo-factorials,
\[
F(z):=\sum_{n\ge0} \alpha_n z^n =1-z-2z^2+2z^3+16{z^4}-40{z^5}-\cdots ,
\]
to be taken in the sense of formal power series since its radius of convergence is~0.
To wit:
\begin{equation}\label{confraca}
F(z)\equiv \sum_{n\ge0} \alpha_n z^n=\cfrac{1}{1+z+\cfrac{3\cdot 1^2\cdot z^2}{
1-z+\cfrac{2^2\cdot z^2}{1+3z+\cfrac{3\cdot 3^2\cdot z^2}{1-3z+\cfrac{4^2\cdot z^2}{\ddots}}}}},
\end{equation}
where the denominators are successively $+1,-1,+3,-3,+5,-5,\ldots$, and 
the numerators
are $3\cdot 1^2,2^2,3\cdot 3^2,4^2,3\cdot 5^2,6^2,\ldots$.
Such a repetitive pattern of order~2 
in what is known as a \emph{Jacobi fraction}
is somewhat unusual.
The fraction~\eqref{confraca} was first discovered experimentally
---proving it in Sections~\ref{sradd-sec}--\ref{proofadd-sec}  is the central theme of the present study.

Finally, the     convergents   of the      continued
fraction~\eqref{confraca}  can   be   made  explicit,   via generating
functions:   this is conducive  to  what seems to be  a   new class of
``elliptic polynomials'' in the sense of Lomont--Brillhart~\cite{LoBr01}; see Section~\ref{ortho-sec}.
In Section~\ref{cong-sec}, we then 
draw several consequences of the previous developments, 
in the form of Hankel determinant evaluations and elementary congruence properties
of pseudo-factorials.

\parag{Why?} As kindly suggested by a referee and by 
the editor of \emph{The Ramanujan Journal},
we offer a few comments relating to the motivations 
behind studies, such as the
present one.

One primary motivation is to gain a better understanding of a still mysterious 
class of continued fractions, the ones that have coefficients that are
polynomial (or rational) functions in the depth~$n$.
This highly general but somewhat impenetrable class of special functions
seems to have been first noticed by Pollaczek in~\cite{Pollaczek56}.
It is piquant to note that Apery's continued fraction~\cite{Poorten79},
which first led to an irrationality proof of Riemann's $\zeta(3)$,
has cubic denominators and sextic denominators, a feature also shared by the
Conrad continued fractions~\cite{Conrad02,CoFl06}
relative to the Dixonian functions $\sm,\cm$,
themselves closely related to pseudo-factorials (Section~\ref{dixon-sec}).
Our main continued fraction (Theorem~\ref{main-thm}) is Pollaczek
in a mildly extended sense: it has polynomial coefficients, 
but modulated by
an unsual periodicity of order~2.

On another register, the relation of elliptic functions to continued fractions is
an old subject, going back at least to Eisenstein, Stieljes, Rogers, and Ramanujan. 
Whereas the $q$--series and $\vartheta$--function aspects 
are not immediately relevant to our discussion, 
we may observe that the continued fractions relative 
to the Jacobian functions $\sn,\cn,\dn$ have arithmetic content:
together with the Hankel determinant product evaluations that they imply,
they have been used by Milne  in his penetrating study~\cite{Milne02} of 
representations of integers as sums of $4n^2$ or $4n(n+1)$ square or triangular numbers.
In this spirit, we offer some new Hankel determinant evaluations in Section~\ref{cong-sec},
but will leave to others the task of determining whether they are of some arithmetic interest.

The class
of orthogonal polynomials associated to any continued fraction, 
whose  coefficients are rational-in-the-depth, is also of interest.
Many instances of low degree have been categorized by Chihara in~\cite{Chihara78}.
 ``Elliptic polynomials'', that is, orthogonal polynomials
associated with series expansions of elliptic functions
have a tradition that goes back at least to Carlitz~\cite{Carlitz60}
and they form the topic of the entire book by Lomont and Brillhart~\cite{LoBr01}.
The discovery of a new class (Section~\ref{ortho-sec}) is 
of interest in this context.

Finally, the pseudo-factorials are tightly coupled with
$\frac{1}{\pi}\Gamma(\frac13)^3$, as seen from our discussion of lattice sums 
in Section~\ref{lat-sec}. In this context,
it is well worth noting a spectacular
arithmetic continued fraction for $\Gamma(\frac13)^3$,
recently obtained by Tanguy Rivoal~\cite{Rivoal09},
which is of the sextic--duodecimic type(!). 
This and previous observations reflect the fact that our understanding
of an orbit of questions, surrounding Pollaczek continued fractions, elliptic functions, elliptic polynomials,
and diophantine approximation properties, is still fragmentary; but
they also suggest that  ``hidden'' structures are
yet to be discovered in this area (see also our brief conclusion in Section~\ref{concl-sec}).

\section{\bf Elliptic connections: Dixonian functions} \label{dixon-sec}

This section serves to establish the first connection between pseudo-factorials
and elliptic functions. The starting point is
the exponential generating function defined by~\eqref{egf};
it satisfies a functional equation,
\begin{equation}\label{funeq}
f'(z)=-f(-z)^2,
\end{equation}
which directly translates the defining recurrence~\eqref{alphadef}.

To make $f(z)$ explicit, take
the functional equation~\eqref{funeq} and differentiate once, so that
\begin{equation}\label{interim0}
f''(z)=2f(-z)f'(-z), \qquad \hbox{implying}\quad
f''(z)=-2\sqrt{-f'(z)}f(z)^2,
\end{equation}
since, by~\eqref{funeq} again,
one has $f'(-z)=-f(z)^2$ and $f(-z)=\sqrt{-f'(z)}$.
In order to solve
the nonlinear differential equation, ``cleverly'' 
multiply by~$\sqrt{-f'(z)}$ to get
\[
f''(z)\sqrt{-f'(z)}=2f(z)^2f'(z),
\]
which is integrated to give
\begin{equation}\label{interim1}
-\frac23 \left(-f'(z)\right)^{3/2}=\frac23 f(z)^3-\frac23 K,
\end{equation}
with~$K$ a yet unspecified constant.
Equivalently, one has
\begin{equation}\label{interim2}
\frac{-f'(z)}{(K-f(z)^3)^{2/3}}=1,
\end{equation}
which upon one more integration gives
\begin{equation}\label{intform1}
\int_{f(z)}^1\frac{dw}{(K-w^3)^{2/3}}=z,
\end{equation}
where use has been made of the initial condition $f(0)=1$.
The constant~$K$ is finally identified by
means of a second order expansion (with
$f'(0)=-1$, $f''(0)=-2$),  to the effect that 
one must have $K=2$. (The computations parallel those of~\cite{CoFl06}.)

In view of our subsequent treatment, it is convenient to
standardize~\eqref{intform1}.
A linear change of variables yields
\begin{equation}\label{intform2}
\int_{2^{-1/3}f(z)}^{2^{-1/3}} \frac{dy}{(1-y^3)^{2/3}}=2^{1/3}z,
\end{equation}
where we have taken into account that~$K=2$.
Throughout this study, a fundamental constant is $\pi_3$ (a period of the 
function~$\sm$ defined below in~\eqref{dix1}),
\begin{equation}\label{perio}
\pi_3:=3\int_0^{1}\frac{dy}{(1-y^3)^{2/3}}=B\left(\frac13,\frac13\right)=
\frac{\Gamma(\frac13)^2}{\Gamma(\frac23)}=
\frac{\sqrt{3}}{2\pi}\Gamma\left(\frac13\right)^3, 
\end{equation}
where the evaluation results from the classical Eulerian integral~\cite[\S12.4]{WhWa27}:
\begin{equation}\label{betadef}
\def\B{\operatorname{B}}
\B(\alpha,\beta):=\int_0^1 t^{\alpha-1}(1-t)^{\beta-1}\, dt
=\int_0^\infty \frac{u^{\alpha-1}}{(1+u)^{\alpha+\beta+1}}\, du
=\frac{\Gamma(\alpha)\Gamma(\beta)}{\Gamma(\alpha+\beta)}.
\end{equation}
(Numerically\footnote{
	The notation~$x\doteq \xi$, with $\xi$ a decimal fraction, indicates that~$\xi$ is an approximation
	of~$x$ to the last digit stated.
}, we find $\pi_3\doteq 5.29991\,62508$.)
A simple computation shows that 
\[\int_0^{2^{-1/3}} \frac{dy}{(1-y^3)^{2/3}} 
=\frac16 B\left(\frac13,\frac13\right) =\frac{\pi_3}{6},
\]
so that we can write
\begin{equation}\label{dix0}
\int_0^{2^{-1/3} f(z)}  \frac{dy}{(1-y^3)^{2/3}}=\frac{\pi_3}{6}-2^{1/3}z.
\end{equation}

The function $f(z)$ can now be expressed in terms of specific
elliptic functions introduced by A.~C. Dixon in his 
original memoir~\cite{Dixon90}. 
Define the function
$\sm(z)$ by the equation
\begin{equation}\label{dix1}
\int_0^{\sm(z)}  \frac{dy}{(1-y^3)^{2/3}}=z.
\end{equation}
Then, a comparison of~\eqref{dix0} and~\eqref{dix1} permits us to identify~$f(z)$.
\begin{theorem}\label{dix-thm} 
The exponential generating function of pseudo-factorials 
satisfies
\begin{equation}\label{dixie}
f(z)=2^{1/3} \sm\left(\frac{\pi_3}{6}-2^{1/3}z\right),
\end{equation}
where the Dixonian elliptic function $\sm(z)$ is as in~\eqref{dix1} and~$\pi_3$ is
the constant~\eqref{perio}.
\end{theorem}

We can offer a few comments regarding Dixonian functions.
There is actually a \emph{pair} of ``higher-order trigonometric''
functions, $\sm$ and $\cm$, where 
$\sm$ and $\cm$ are evocative of a sine and a cosine function, respectively.
Their properties can be developed
from first principles, as done by Dixon followed by Conrad--Flajolet~\cite{Conrad02,CoFl06,Dixon90}, starting with the differential system,
\begin{equation}\label{smcm}
\sm'(z)=\cm(z)^2, \qquad
\cm'(z)=-\sm(z)^2,
\end{equation}
and initial conditions~$\sm(0)=0$, $\cm(0)=1$. 
(See also the works of
Lundberg~\cite{Lundberg79} and the 
recent developments by Lindquist and Peetre~\cite{LiPe01b,LiPe01}
for a yet more general approach.)
For the record, we note that
\[
\sm(z)=z-4\frac{z^4}{4!}+160\frac{z^7}{7!}-20800\frac{z^{10}}{10!}-\cdots,
\quad
\cm(z)=1-2\frac{z^3}{3!}+40\frac{z^6}{6!}-3680\frac{z^9}{9!}+\cdots\,,
\]
whose coefficients ($1,-4,160,\ldots$ and $1,-2,40,\ldots$)
are respectively~{\bf A104133} and {\bf A104134} of Sloane's \emph{Encyclopedia}.
The works of  Conrad and Flajolet~\cite{Conrad02,CoFl06}
provide continued fraction expansions for the ordinary generating function of 
these coefficients, but these are then relative to the expansions of~$\sm,\cm$ at~$0$,
and \emph{not} at~$\pi_3/6$, as in~\eqref{dixie}.
Finally, we observe that, by the calculation~\eqref{interim1} and by~\eqref{smcm},
we have the identity
\[
\sm(z)^3+\cm(z)^3=1,
\]
so that the pair $(\sm(z),\cm(z))$ parametrizes the \emph{Fermat cubic} $\bf F_3$ defined by
the equation $X^3+Y^3=1$, which is of genus~1.

\section{\bf Elliptic connections: Weierstra{\ss} forms} \label{weier-sec}

It is \emph{a priori} possible to reduce the exponential generating function $f(z)$
of pseudo-factorials to any
of the several canonical forms of elliptic functions. Here, we show, by elementary calculations
similar to the ones of the previous section, how to arrive directly at 
an expression involving the Weierstra{\ss} function $\wp$.
We recall that this function $\wp(z)\equiv\wp(z;g_2,g_3)$
is classically defined by the 
nonlinear differential equation~\cite{WhWa27}
\begin{equation}\label{defwp}
\wp'(z)^2=4\wp(z)^3-g_2\wp(z)-g_3,
\end{equation}
together with the initial condition $\wp(z)\sim z^{-2}$ as $z\to0$. 
By design, the pair $(\wp(z),\wp'(z))$ parametrizes the elliptic curve 
$Y^2=4X^3-g_2X-g_3$, with invariants $g_2,g_3$.

The starting point is
the fundamental relation~\eqref{funeq}, namely, 
 $f'(z)=-f(-z)^2$.
We first claim the identity
\begin{equation}\label{E1}
f(z)^3+f(-z)^3=2.
\end{equation}
The proof is obtained by verifying that the derivative 
of the left-hand side is~0,
\[
\left(f(z)^3+f(-z)^3\right)'=
3f(z)^2f'(z)-3f(-z)^2f'(-z)=0,\]
(the final reduction uses~\eqref{funeq}), 
combined with the initial condition~$f(0)=1$.
Next, we have
\begin{equation}\label{E2}
\left(-f(z)f(-z)\right)'=f(-z)^3-f(z)^3,
\end{equation}
again by way of~\eqref{funeq}.

Now set 
$$g(z):=-f(z)f(-z)=-1+3z^2-3z^4+3z^6-\frac{18}{7}z^8+\frac{15}{7}z^{10}-\frac{12}{7}z^{12}+\cdots\,.$$
 The basic elliptic connection 
is provided by the differential relation
\begin{equation}\label{W0}
g'(z)^2=4g(z)^3+4,
\end{equation} 
which is clearly of the Weierstra{\ss} type~\eqref{defwp}.
To see it, it suffices to square the two sides of
the identities~\eqref{E1} and~\eqref{E2}, then compare
the outcomes.
Equation (20) then shows that $g(z)$ is closely related to 
the elliptic curve
$\mathcal E$ defined by
\begin{equation}\label{ell0}
Y^2=4X^3+4\ .
\end{equation}
The curve $\mathcal E$ contains six integral points\footnote{
	These are the only rational points of the curve~$\cal E$,
	since it is known that the Mordell curve $Y^2=X^3+1$ has six rational points;
	see, for instance, the SIMATH tables that are accessible 
	at {\tt tnt.math.metro-u.ac.jp/simath/MORDELL/MORDELL+}.
} (including the point
at infinity corresponding to the identity element of the underlying group) 
forming a cyclic group of order six. The non-trivial points of this
group are: $(-1,0)$ (of order $2$), $(0,\pm 2)$ (of order $3$)
and $(2,\pm 6)$ (of order $6$).
Since $(g(0),g'(0))=(-1,0)$, the series $g(z)$ represents the 
expansion  of the Weierstra{\ss}  function $\wp$
around the unique real $2$-torsion point $(-1,0)$ of 
$\mathcal E$.

The following result  recovers $f(z)$ from $g(z)\equiv -f(z)f(-z)$
and constitutes  the main result of this section.
\begin{theorem}\label{weier-thm}
Let $\wp(z):=\wp(z;0,-4)$ be the Weierstra{\ss} function with invariants $g_2=0$
and $g_3=-4$ and smallest positive real period\footnote{%
	The reader should be warned that the six numbers
	$\pm 6r,\pm 6re^{\pm i\pi/3}$ are \emph{not} the shortest non-zero elements of
	the period lattice for $\wp$. They generate a sublattice of index $3$
	in the period lattice of $\wp$ whose shortest elements are given by 
	the six numbers $\pm 2 i\sqrt{3}r,\ \pm2 i\sqrt{3}re^{\pm i\pi/3}$; see 
	also Section~\ref{lat-sec} and Figure~\ref{lat-fig}.
} 
\begin{equation}\label{perio2}
6r=\pi_32^{-1/3}=\frac{2^{-4/3}3^{1/2}}{\pi}\Gamma\left(\frac13\right)^3,
\end{equation}
with~$\pi_3$ as in~\eqref{perio}.
The exponential generating function of pseudo-factorials satisfies
\begin{equation}\label{wconnex}
f( i\sqrt{3}z)=\frac{-\wp'(z+3r)-2 i\sqrt{3}}{2 i\sqrt{3}\wp(z+3r)}.
\end{equation}
\end{theorem}
\noindent
\begin{proof}
We first establish the expression~\eqref{perio2} of the real period of $\wp(z)$,
only making use of the most basic properties of elliptic functions~\cite[Ch.~XX]{WhWa27}.
Let us denote  temporarily the \emph{real half-period} by~$\varpi$. 
By general properties of elliptic functions\footnote{%
	We have $\wp(\varpi)=0$ since $\wp'$ is odd; hence $\wp(\varpi)$ must be a root
	of the third-degree polynomial associated with~$\wp$; here, $\wp(\varpi)=-1$.
}, we have $\wp'(\varpi)=0$, while $\wp(\varpi)$ is a real root of $4w^3+4=0$;
that is, $\wp(\varpi)=-1$. Thus, since $\wp$ is the inverse of an elliptic integral, we must have
\[\def\B{\operatorname{B}}\renewcommand{\arraystretch}{1.9}
\begin{array}{lll}
\varpi&=& \ds \int_{-1}^\infty \frac{dw}{\sqrt{4w^3+4}}
=\frac12\int_{-1}^0 \frac{dw}{\sqrt{w^3+1}}+\frac12\int_0^\infty \frac{dw}{\sqrt{w^3+1}}\\[1.5truemm]
&=&\ds \frac{1}{6}\B\left(\frac13,\frac12\right)+\frac{1}{6}\B\left(\frac13,\frac16\right)=
{2^{-7/3}3^{1/2}}{\pi}^{-1}\Gamma\left(\frac13\right)^3,
\end{array}
\]
as shown by the changes of variables $t=-w^3$ and $u=w^{3}$ (respectively) in 
the last two integrals of the first line, followed by Eulerian Beta function evaluations~\eqref{betadef}.
We henceforth denote the quantity~$\varpi$ by~$3r$.

The function $f( i\sqrt{3}z)$ is determined by
$f(0)=1$ and by the functional equation deduced from~\eqref{funeq}:
\begin{equation}\label{z1}
\frac{1}{ i\sqrt{3}}\frac{d}{dz}f( i\sqrt{3}z)=-f(- i\sqrt{3}z)^2\,.
\end{equation}
We proceed to verify~\eqref{wconnex}. Since $\wp'(3r)=0$ and~$\wp(3r)=-1$, we first have
$$\frac{-\wp'(3r)-2 i\sqrt{3}}{2 i\sqrt{3}\wp(3r)}=f(0)=1\,.$$
In view of~\eqref{z1}, the equality~\eqref{wconnex} then reduces to proving the identity
\begin{equation}\label{iszero}
\frac{1}{ i\sqrt{3}}\left(\frac{-\wp'(z+3r)-2 i\sqrt{3}}{2 i\sqrt{3}\wp(z+3r)}\right)'+
\left(\frac{-\wp'(-z+3r)-2 i\sqrt{3}}{2 i\sqrt{3}\wp(-z+3r)}\right)^2=0\,.
\end{equation}
Since $\wp(z)$ is an even  function, which is $6r-$periodic,
the left-hand side of~\eqref{iszero} can be 
put under the rational form~$A/B$, where the numerator~$A$ involves $\wp,\wp',\wp''$ at~$z+3r$.
Writing $\wp,\wp',\wp''$ for $\wp(z+3r),\wp'(z+3r),\wp''(z+3r)$, it remains to
verify that 
$$A=-2\wp''\wp+2(\wp'+2 i\sqrt{3})\wp'+\wp'^2-4 i\sqrt{3}\wp'-12$$
vanishes identically. Derivation of the differential equation 
$\wp'^2=4\wp^3+4$ for $\wp$ yields $\wp''=6\wp^2$
and we indeed obtain
$$A=-12\wp^3+3\wp'^2-12=0\,,$$
which concludes the proof of the statement.
\end{proof}

\section{\bf Lattice-sum expressions for the pseudo-factorials $\alpha_n$} \label{lat-sec}
From the Dixonian  as well as the Weierstra{\ss} connections discussed in
the previous section,
expressions of the $\alpha_n$ as \emph{lattice sums} can be developed.

\begin{theorem}\label{lat-thm}
The pseudo-factorials are expressible as lattice sums involving
a twelfth root of unity: with $\rho=2r\sqrt{3}$ and~$r$ as in Eq.~\eqref{perio2}, one has, for any~$n\ge2$:
\begin{equation}\label{latsum}
\alpha_n = - \frac{n!}{\rho^{n+1}}\sum_{\lambda,\mu\in\mathbb{Z}}
\frac{\zeta^{8\lambda+4\mu}}{\left[\left(\lambda-\frac12\right)\zeta
+\left(\mu-\frac12\right)\zeta^{-1}\right]^{n+1}} \,,
\qquad \zeta:=e^{i\pi/6}.
\end{equation}
\end{theorem}

\noindent
The formula~\eqref{latsum} 
implies
explicit asymptotics for~$(\alpha_n)$, namely,
\[
\frac{\alpha_{2\nu }}{(2\nu )!} \sim   (-1)^\nu  3^{-\nu } r^{-2\nu -1} 
\quad\hbox{and}\quad
\frac{\alpha_{2\nu +1}}{(2\nu +1)!}
 \sim   (-1)^{\nu +1} 3^{-\nu -1}r^{-2\nu -2} , 
\]
once only the relevant dominant poles  are retained.
This explains in particular the regular sign pattern ``{\sf +\,-\,-\,+}'' observed
in~\eqref{inivals}.

\begin{proof}
Let~$G(z)$ be an elliptic (i.e., meromorphic, doubly periodic) function that 
has only simple poles. Let~$\Lambda=\Z\omega_1\oplus \Z\omega_2$ be its lattice 
of periods and let~$\cal S_0$ be the set of poles
contained in a fundamental domain of~$\Lambda$. Then, if
$0\not\in\cal S_0$, one has
\begin{equation}\label{classell}
G(z)=G(0)+zG'(0)+\sum_{\omega\in\Lambda}
\sum_{\rho\in\cal S_0}
r_{\rho} 
\left[\frac{1}{z-(\rho+\omega)}+\frac{1}{\rho+\omega}+\frac{z}{(\rho+\omega)^2}\right],
\end{equation}
where $r_s$ represents the residue of~$G(z)$ at~$z=s$.
Theorem~\ref{weier-thm} and the formula~\eqref{classell}
show that it suffices,
up to an affine transformation, to work out the singular structure of
\[
H(z)=\frac{-\wp'(z)-2i\sqrt{3}}{2i\sqrt{3}\wp(z)}
\]
in order to deduce the partial fraction expansion of~$f(z)$, from which 
the lattice sum~\eqref{latsum} expression will result.

The function~$\wp(z)\equiv\wp(z;0,-4)$ has lattice of
periods
\begin{equation}\label{lath}
(\Z e^{i\pi/6}\oplus \Z e^{-i\pi/6})2\sqrt{3}r
\end{equation}
and~$H(z)$  has  simple poles at~$z=0$ and~$z=\pm  2r$. Given
the series expansion  $\wp(z)=z^{-2}+O(z^2)$ of~$\wp$ around~0,  we
find the  residue  of~$H(z)$ at~$0$ to be
\begin{equation}\label{resh0}
\Res(H(z);z=0)=\frac{-i}{\sqrt{3}}.
\end{equation}
By the discussion         following~\eqref{ell0},   we   also            have
$(\wp(\pm2r),\wp'(2r))=(0,\pm2)$.       Moreover,   we  have
$(\wp(3r),\wp'(3r))=(-1,0)$                                        and
$(\wp'(3r),\wp''(3r))=(\wp'(3r),6\wp(3r)^2)=(0,6)$,     hence
$\wp'(2r)=-2$ and $\wp'(4r)=\wp'(-2r)=2$. The expansion of~$\wp(z)$ around~$2r$ is thus
\[
\wp(z+2r)=-2z+2z^4-\frac{8}{7}z^7+\frac{4}{7}z^{10}+\cdots,
\]
from which a simple computation provides the residue of~$H(z)$ at~$2r$,
and, similarly, the residue at~$-2r$:
\begin{equation}\label{resh1}
\Res\left(H(z);z=2r\right)=\frac{3+i\sqrt{3}}{6},
\qquad
\Res\left(H(z);z=-2r\right)=\frac{-3+i\sqrt{3}}{6},
\end{equation}

\begin{figure}\small
\vspace*{2truemm}
\begin{center}
\setlength{\unitlength}{0.95truecm}
\begin{picture}(6,6)
\thicklines
\put(0,0){\Img{5.7}{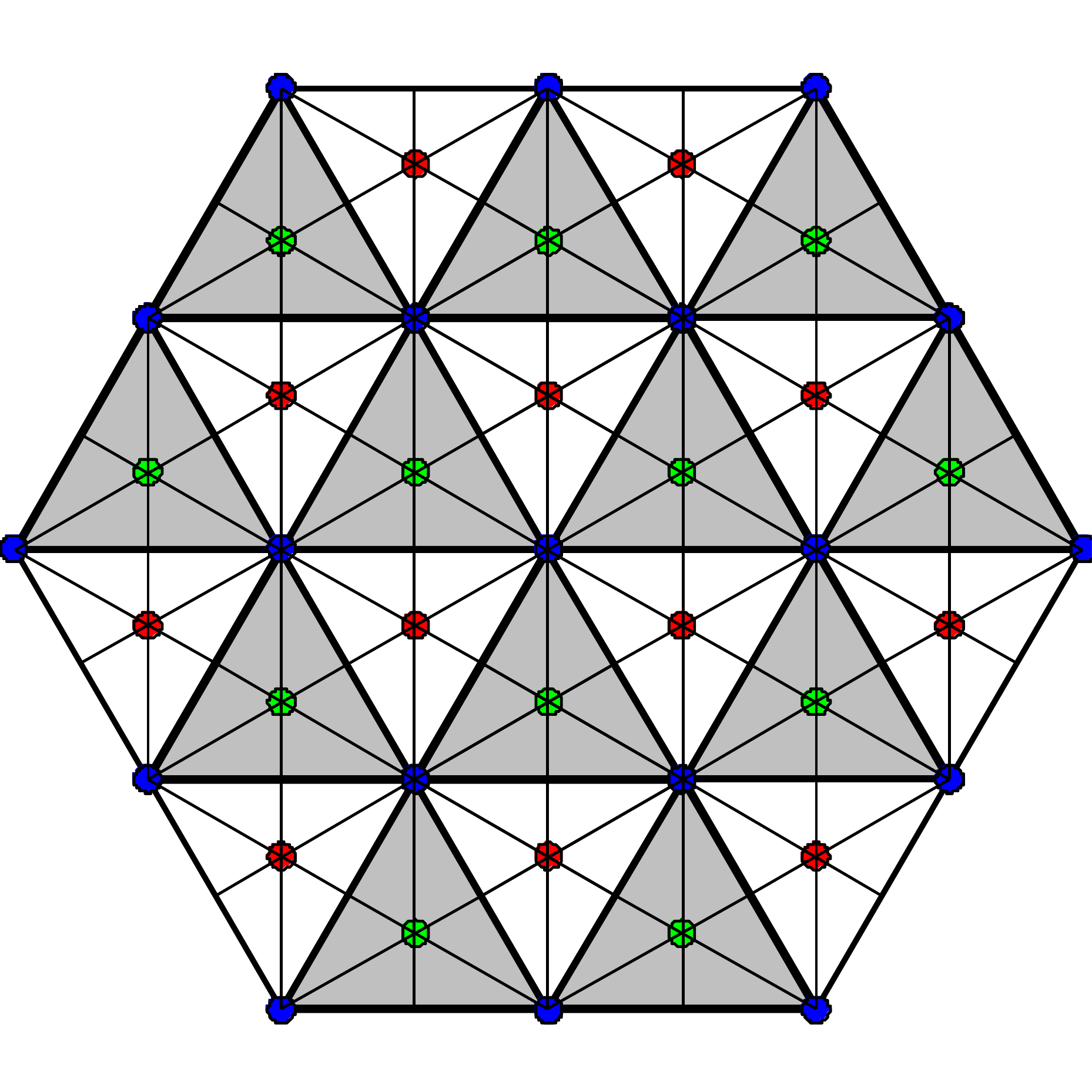}}
\put(-0.4,3.0){\vector(1,0){7.1}}   
\put(6.15,3.2){$\Re(z)$}
\put(3.75,0.0){\vector(0,1){6.0}}  
\put(3.77,0.0){\vector(0,1){6.0}}  
\put(2.9,5.8){$\Im(z)$}
\end{picture}
\qquad
\setlength{\unitlength}{0.8truecm}
\raisebox{5.15truemm}{\begin{picture}(6,6.4)
\put(0.5,3){\vector(1,0){5.1}}   
\put(5.9,3){$\Re(z)$}
\put(3,-0.2){\vector(0,1){6.6}}  
\put(2.0,6.1){$\Im(z)$}
\put(3,3){\circle*{0.15}} 
\put(3.1,3.1){$0$}
\put(1.3,3){\circle*{0.2}} 
\put(1.3,3){\circle{0.35}}
\put(0.45,3.1){$-3r$}
\put(4.7,3){\circle*{0.2}}
\put(4.8,3.1){$+3r$}
\thicklines
\put(1.3,3){\line(3,5){1.7}}
\put(1.3,3){\line(3,-5){1.7}}
\put(4.7,3){\line(-3,5){1.7}}
\put(4.7,3){\line(-3,-5){1.7}}
\put(3,5.83){\circle*{0.2}}
\put(3.4,5.83){$+i3\sqrt{3}r$}
\put(3,0.17){\circle*{0.2}}
\put(3.4,0.17){$-i3\sqrt{3}r$}
\put(3,3.94){\circle*{0.2}}
\put(3,3.94){\circle{0.35}}
\put(3.4,3.94){$+i\sqrt{3}r$}
\put(3,2.06){\circle*{0.2}}
\put(3,2.06){\circle{0.35}}
\put(3.4,2.06){$-i\sqrt{3}r$}
\end{picture}}
\end{center}
\caption{\label{lat-fig}\small
\emph{Left:} the ``primary'' lattice of periods [thick lines] of~$f(z)$,
where a fundamental domain is obtained by the union of two 
adjacent equilateral triangles (one grey, 
one white); the poles (three per fundamental domain)
are represented by small discs and form a ``secondary'' hexagonal lattice [thin lines]. 
\emph{Right}: a diagram showing
the three poles (circled) of a fundamental domain of~$f(z)$ around the origin.
}
\end{figure}

Now, by~\eqref{lath}, \eqref{resh0}, and~\eqref{resh1},
the singular structure of~$H(z)$ is entirely known. 
Since $f(iz\sqrt{3})=H(z+3r)$,
 an affine transformation 
(composed of a translation, a rotation, and a dilation) 
provides the singular structure of~$f(z)$ itself---we 
abbreviate the discussion, which is routine.
As represented in Figure~\ref{lat-fig},
the lattice of periods~$\Lambda$ of~$f(z)$
is a hexagonal lattice with generators $3r(-1\pm i\sqrt{3})$;
we may call it the ``primary'' hexagonal lattice.
There are three poles of~$f(z)$ in the fundamental domain, at~$-3r$ and $\pm i\sqrt{3}r$,
which, upon translation by~$\Lambda$, generate a ``secondary'' hexagonal lattice.
The residue of~$f(z)$ (deduced from~\eqref{resh0} and~\eqref{resh1}),
is then found to be of the form $\zeta^{8\lambda+4\mu}$ 
at a point of the form $2r\sqrt{3}[(\lambda-1/2)\zeta+(\mu-1/2)\zeta^{-1}]$,
which corresponds to a 3--colouring of the secondary hexagonal lattice
(since~$\zeta^4$ is a third root of unity).
The corresponding partial fraction decomposition~\eqref{classell}
results for~$f(z)$. Finally, the fact\footnote{
Notation: we let $[z^n]f$ represent the coefficient of~$z^n$ in the formal power series 
or analytic function~$f$.} that, for $n\ge2$,
\[
[z^n] \left(\frac{1}{z-a}+\frac{1}{a}+\frac{z}{a^2}\right) = -\frac{1}{a^{n+1}}.
\]
yields the stated lattice-sum formula for~$\alpha_n$.
\end{proof}

The sum~\eqref{latsum} establishes the pseudo-factorials
as a two-dimensional analogue of Bernoulli and Euler numbers,
one that is relative to the hexagonal lattice. 
It might be of interest to investigate systematically continued fractions 
relative to other hexagonal lattice sums. It is worthy of note that 
arithmetic properties of analogous 
``lemniscatic'' sums,  relative to the square lattice, have been studied by 
Hurwitz~\cite{Hurwitz97,Hurwitz99}.

\section{\bf The Stieltjes--Rogers addition theorem and continued fractions} \label{sradd-sec}

This section serves to     introduce the basic technology  needed   to
develop an explicit continued  fraction   expansion from an   addition
theorem of a    suitable form.  An experimental  approach specialized to
pseudo-factorials follows, in  Section~\ref{experadd-sec}.
We can then reap the   crop  in Section~\ref{proofadd-sec} and finally  prove our
main continued fraction result (Theorem~\ref{main-thm}).

Stieltjes and Rogers independently discovered 
that the continued fraction expansion of an ordinary generating function,
$\Phi(z)=\sum \phi_nz^n$, is closely related to \emph{addition formulae} satisfied by
the corresponding exponential generating function, $\phi(z)=\sum \phi_n z^n/n!$.
First, a definition.

\begin{definition}
Let $\phi(z)=1+\sum_{n\ge1} \phi_n z^n/n!$ be a formal power series.
It is said to satisfy an \emph{addition formula of the Stieltjes--Rogers type} if
there exist nonzero constants $\omega_1,\omega_2,\ldots$ and formal power series
$\varphi_0(z),\varphi_1(z),\ldots$, such that
\begin{equation}\label{sradd}
\phi(x+y)=\varphi_0(x)\varphi_0(y)+\omega_1\varphi_1(x)\varphi_1(y)+\omega_2\varphi_2(x)\varphi_2(y)+
\cdots,
\end{equation}
where $\varphi_r(z)$ has valuation~$r$ and is normalized by
$\varphi_r(z)=(z^r/r!)+O(z^{r+1})$.
\end{definition}
In~\eqref{sradd}, the valuation condition on $\varphi_r$ is essential,
the normalization $\phi_r(z)\sim z^r/r!$ being a mere convenience
for what follows. 

The  addition    formula gives rise  to an algorithm  
for computing the $\varphi_\ell(z)$, knowing~$\phi(z)$, either
symbolically or via some series expansion.
First, setting $y=0$ in the addition formula
shows that $\phi(z)$ must be equal to $\varphi_0(z)$
(this makes use of the normalization~$\varphi_0(0)=1$).
Next, assume  that the  functions   $\varphi_0,\ldots,\varphi_{\ell-1}$ 
and the coefficients $\omega_1,\ldots,\omega_{\ell-1}$ have already  been
determined.  Then,   
by  differentiating~$\ell$  times   the  addition
formula~\eqref{sradd} with respect to~$y$, then setting~$y=0$, one finds 
\[
 \partial_x^\ell \phi(x)-\sum_{j=0}^{\ell-1}\omega_{j}
\varphi_j(x)  \left[\partial_y^\ell \varphi_j(y)\right]_{y=0}
=\omega_\ell\varphi_\ell(x).\] 
Upon comparing the coefficient of
$x^\ell$ in the Taylor     expansions of both sides,  
we see that \emph{at  most}     one nonzero
coefficient~$\omega_\ell$   can   satisfy the equation,
given the normalization $\varphi_\ell(x)=x^\ell/\ell!+O(x^{\ell+1})$.
 Once  this choice has
been fixed,  then $\varphi_{\ell}(x)$  is uniquely
determined as a linear combination
of the previous~$\varphi_j(z)$, together with derivatives of~$\phi$,
and the process can continue. 
This construction also determines 
the broad class of functions in which the $\varphi_\ell(z)$ live:
\emph{each $\varphi_\ell(z)$ belongs to the vector space generated over~$\C$ 
by the first  derivatives $\phi,\phi^{(1)},\ldots,\phi^{(\ell)}$ of the function~$\phi(z)$}.
 This algorithm, albeit suboptimal from a computational point of view, 
permits us to experiment with addition formulae relative to the generating function of any given 
number sequence~$(\phi_n)$, a fact that  will  prove especially valuable in Section~\ref{experadd-sec}.

Addition formulae of the Stieltjes--Rogers type
are logically equivalent to
continued fraction expansions
as expressed by the following central theorem
originally due to Stieltjes~\cite{Stieltjes89} and Rogers~\cite{Rogers07};
see also Perron~\cite[p.~133]{Perron57} and Wall~\cite[p.~204]{Wall48}.

\begin{theorem}[Stieltjes--Rogers] \label{sradd-thm}
 Let the \emph{exponential} generating function~$\phi(z)=1+\sum_{n\ge1} \phi_n z^n/n!$ 
satisfy an addition
formula of the form~\eqref{sradd}. Then, the corresponding \emph{ordinary} generating function 
$\Phi(z)=1+\sum_{n\ge1}\phi_n z^n$ admits a 
Jacobi-type continued fraction expansion\footnote{%
	Such an expansion is diversely known as a Jacobi fraction,
	a $J$--fraction, or an associated continued fraction.
},
\begin{equation}\label{Jfrac}
\Phi(z)=\cfrac{1}{1-c_0z-\cfrac{a_1z^2}{1-c_1z-\cfrac{a_2z^2}{\ddots}}},
\end{equation}
where the coefficients are determined by
\[
a_j=\frac{\omega_j}{\omega_{j-1}}\quad (j\ge1),
\qquad
c_j=\varphi_{j,j+1}-\varphi_{j-1,j}\quad (j\ge0).
\]
There, $\varphi_{j,k}=k![z^k]\varphi_j(z)$, $\varphi_{-1,k}=0$, and $\omega_0=1$.
\end{theorem}

%

As an illustration, following Stieltjes and Rogers, we examine the case of
\[
\phi(z)=\sec(z)\equiv\frac{1}{\cos(z)}.\]
Each $\varphi_k(z)$ (provided it exists) must then \emph{a priori}
be of the form 
$\sec(z)P_k(\tan(z))$, where $P_k$ is a polynomial satisfying
$\deg P_k=k$. In a simple case like this,
classical trigonometric identities yield
\[
\sec(x+y)=\frac{1}{\cos(x)\cos(y)-\sin(x)\sin(y)}
=\sum_{k\ge0} \sec(x)\tan(x)^k\cdot \sec(y)\tan(y)^k,\]
which, in normalized form, becomes
\begin{equation}\label{secadd}
\sec(x+y)=\sum_{k\ge0} (k!)^2 \left(\sec(x)\frac{\tan(x)^k}{k!}\right)
\cdot \left(\sec(y)\frac{\tan(y)^k}{k!}\right).
\end{equation}
With $E_{n}=n![z^n]\sec(z)$ an Euler number, Theorem~\ref{sradd-thm} then 
provides the continued fraction expansion:
\[
\sum_{n\ge0} E_n z^n =\cfrac{1}{1-\cfrac{1^2\cdot z^2}{1-\cfrac{2^2\cdot z^2}
{\ddots}}},\]
where the coefficients $1^2,2^2,3^2,\ldots$, are obtained here as quotients of
consecutive squared factorials. The absence of linear terms
reflects the parity of~$\sec(z)$.

\section{\bf Experimental determination  of the addition formula for $f(z)$}\label{experadd-sec}

Section~\ref{sradd-sec} has shown that, in order to approach the construction of a continued
fraction expansion for pseudo-factorials, we need to develop a suitably constrained 
addition formula for their exponential generating function~$f(z)$, which is elliptic.
We proceed here in an experimental manner in order to infer the likely shape
of such an addition formula\footnote{%
	This section is not, strictly speaking, necessary. It could
	have been replaced by the shorter but somewhat obscure
	formulation: \emph{``Crystal ball gazing revealed to us
	the addition formula~\eqref{addf}}.''
}. Once this had been done, the proof of our main continued fraction
reduces to purely mechanical verifications to be carried out in the next section.

A first idea that comes to mind is to look for an addition formula
of a kind similar to the secant case~\eqref{secadd}, namely,
\begin{equation}\label{simpadd}
\phi(x+y)=\phi(x)\phi(y)\Psi(\sigma(x)\cdot\sigma(y)),
\end{equation}
for some function (or power series) $\Psi(w)$. 
However, all such formulae can only arise from
a class of special functions
that comprises \emph{five} parametrized subclasses,
of which prototypes are
\[
\sec(z),\quad \frac{1}{1-z},\quad e^{e^z-1},\quad e^{z^2/2},\quad
\frac{1}{2-e^z}.
\]
(This is a rephrasing of a classification of orthogonal
polynomial systems due to Meixner~\cite{Chihara78,Meixner34}.)
Obviously, elliptic functions are not amongst this group.

Another source of   inspiration is  a  continued fraction  
relative to  elliptic functions, which 
is also due to Stieltjes and Rogers.
 With $\sn,\cn,\dn$  the Jacobian elliptic functions,
as classically defined, where we leave the modulus~$k$ implicit, we have
\begin{equation}\label{cnadd}
\cn(x+y)=\frac{\cn x \cn y -\sn x \sn y \dn x \dn y}
{1-k^2\sn^2 x \sn^2 y},
\end{equation}
see for instance~\cite[p.~497]{WhWa27}.
This can be put into an equivalent Stieltjes--Rogers form
(up to normalization), namely,
\begin{equation}\label{cnxy}
\cn(x+y)=\cn x\cn y \left(1-\sn x  \dn x \sn y \dn y
+k^2 \sn x \cn^2 x \dn x \sn y \cn^2 y \dn y+\cdots\right),
\end{equation}
corresponding to the continued fraction expansion
\[
\sum_{n\ge0} \cn_n z^n=\cfrac{1}{1-\cfrac{1^2\cdot z^2}{1-\cfrac{2^2 k^2\cdot z^2}{
{1-\cfrac{3^2\cdot z^2}{\ddots}}}}},
\]
where $\cn_n:=n![z^n]\cn (z)$.

We now turn to the continued fraction expansion relative to pseudo-factorials
 which, by Theorem~\ref{sradd-thm}, involves determining the right addition formula
for~$f(z)$.
Based on experiments under the Maple system as well as on induction from 
the secant~\eqref{secadd} and  Jacobian~\eqref{cnadd} cases,
we started searching for an addition formula of the form
\begin{equation}\label{conj}
f(x+y)=f(x)f(y)\Psi(\sigma(x)\sigma(y))
+h(x)h(y)\Xi(\tau(x)\tau(y)),
\end{equation}
for some power series $\Psi,\Xi,h,\sigma,\tau$,
with (at least) $\Psi_0=\Xi_0=1$, and $\sigma(x),\tau(x),h(x)$ all 
being $O(x)$.
Since some binary pattern is present in the continued fraction, it is natural 
further to suppose that
$\sigma(x)=O(x^2)$, $\tau(x)=O(x^2)$, which then
corresponds to an ``odd-even'' addition formula,
\begin{equation}\label{conj0}
f(x+y)=f(x)f(y)+h(x)h(y)+\Psi_1 f(x)f(y)\sigma(x)\sigma(y)
+\Xi_1 h(x)h(y)\tau(x)\tau(y)+\cdots\,,
\end{equation}
where $\Psi_m=[w^m]\Psi(w)$ and~$\Xi_m=[w^m]\Xi(w)$.
In the notations of ~\eqref{sradd}, 
we thus hope for an addition formula of the form
\begin{equation}\label{conj1}
\varphi_{2j}(x)\propto f(x) \sigma(x)^j,
\qquad
\varphi_{2j+1}(x)\propto h(x) \tau(x)^{j-1},
\end{equation}
where $a(x)\propto b(x)$ means that the ratio $a(x)/b(x)$ is a constant.

The pleasant feature of the conjectured expansions~\eqref{conj} and~\eqref{conj1}
is that their plausibility can be effectively tested.
Indeed, from the previous section, we have available an algorithm that 
can determine the (unique)~$\varphi_j(x)$ 
corresponding to $\phi(z)\equiv f(z)$,
this to any desired precision. It then suffices to check that
\[
\frac{\varphi_2(x)}{\varphi_0(x)}
\propto \frac{\varphi_4(x)}{\varphi_2(x)}  \propto\cdots,
\quad\hbox{and}\quad
\frac{\varphi_3(x)}{\varphi_1(x)}
\propto \frac{\varphi_5(x)}{\varphi_3(x)}  \propto\cdots\,.
\]
Verification of these relations for about a dozen of the $\varphi_j$
and till orders in the range 50--100 convinces us that 
we are on the right tracks. 

In fact, we found experimentally that $\sigma(z)=\tau(z)$ up to $O(z^{50})$,
the series starting as
\[
\sigma(z)=3\left(z^2-z^4+z^6-\frac{6}{7}z^8+\frac{5}{7}z^{10}-\cdots\right).
\]
Also, the function $h(z)$ that appears in~\eqref{conj0} must be proportional to~$\varphi_1(z)$
of the addition formula, whose expansion starts as $\varphi_1(z)=z-z^3+\frac14 z^4-\cdots$.
That function $\varphi_1(z)$ must itself, on general grounds, be a linear combination of $f(z)$ and
$f'(z)$ without constant term, so that
\[
\varphi_1(z)=-\frac{1}{3}\left(f(z)+f'(z)\right) \quad\hbox{and}\quad h(z)\propto \varphi_1(z).
\]
Finally, assuming the observed law of the
coefficients in the continued fraction to hold forever,
we can deduce the only possible shape of the $\Xi$ and $\Psi$ functions.
The function~$\sigma$ is then inferred on the basis of the fact that
$\varphi_2(z)\propto \varphi_0(z)\sigma(z)$ must be a linear combination of $f,f',f''$.
All in all, every ingredient of
an addition formula of type~\eqref{conj} is in place,
and we are eventually led to conjecturing an
addition formula for~$f(z)$
\begin{equation}\label{addf}
\left\{\begin{array}{l}
f(x+y)~=~\ds 
\frac{f(x)f(y)-\frac13 h(x)h(y)}
{1-\frac13\sigma(x)\sigma(y)}\\
h(z)~=~ f(z)+f'(z), \qquad 
\sigma(z)~=~1-f(z)f(-z),
\end{array}\right.
\end{equation}
which we shall establish in the next section.

\section{\bf Proof of the continued fraction expansion}\label{proofadd-sec}

At this stage, we know that establishing the continued fraction~\eqref{confraca}
relative to the ordinary generating function~$F(z)$ of pseudo-factorials
reduces to deciding the validity of the conjectured addition formula~\eqref{addf} for~
the exponential generating function~$f(z)$.
The proof we propose is a computer-assisted verification. 
As we shall explain, it only involves
routine algebraic manipulations\footnote{%
	The validity of intermediate steps is, in addition, easily cross-checked by
	means of Taylor series expansions.
}; namely,
rational function operations, normalizations, substitutions, as well as
multivariate polynomial divisions.  The calculations were performed
using the {\sc Maple} computer algebra engine (version~11). 
Without any attempt
at optimization (we purposely wanted our program to rely solely on the most basic 
algebraic operations), the mechanical verification
reduces to the mere execution of a few billion
machine instructions---currently, just a few seconds of elapsed time.

\begin{proposition}\label{add-prop}
The function $f(z)$ satisfies
the following addition formula:
\begin{equation}\label{mainadd}
f(x+y)=\frac{f(x)f(y)-\frac13(f(x)+f'(x))(f(y)+f'(y))}{1-\frac13
(1-f(x)f(-x))(1-f(y)f(-y))}.
\end{equation}
\end{proposition}
\begin{proof}

We can \emph{a priori} appeal to either the Dixon or the Weierstra{\ss} 
framework, and we have opted for the latter. The Weierstra{\ss} $\wp$--function, 
$\wp(z)\equiv \wp(z;0,-4)$ satisfies the two algebraic relations
\[
\begin{array}{lllll}
\wp'(z)^2 &=& \ds 4\wp(z)^3+4 &\qquad& \hbox{(DEF)}
\\\wp(u+v)&=& \ds \frac{1}{4}
\left(\frac{\wp'(u)-\wp'(v)}{\wp(u)-\wp(v)}\right)^2-\wp(u)-\wp(v)
 && \hbox{(ADD).}
\end{array}
\]
The first one (DEF) is the basic differential equation, which serves as \emph{definition}
of~$\wp$;
the second  one (ADD) is  the  familiar  \emph{addition} theorem of  elliptic
function theory~\cite[p.~441]{WhWa27}. Both are ``known'' to {\sc Maple};
both can be viewed as deterministic rewrite rules  permitting one to 
expand and simplify expressions involving~$\wp$.

Let~$6r$ be the fundamental constant (real period) of
Section~\ref{weier-sec}. We know that $(\wp(3r),\wp'(3r))=(-1,0)$.
The addition rule (ADD) combined with the expression of $f(z)$
stated in Theorem~\ref{weier-thm}, Equation~\eqref{wconnex},
then mechanically expresses $f(z)$ as a rational fraction
in~$\wp(Z)$ and~$\wp'(Z)$, where $Z=z/ i\sqrt{3}$. Similar expressions are obtained 
for $f'(z)$ (by standard derivation rules combined with partial reductions by (DEF))
and $f(-z)$ (since $\wp$ is an even function, while $\wp'$ is odd).
In this way, one automatically obtains  rational forms in $\wp,\wp'$ for 
\[
f(x),~f'(x),~f(y),~f'(y),~f(-x),~f(-y),~f(x+y),
\]
where the last one necessitates a substitution $z\mapsto x+y$, followed by
an application of the (ADD) rule.

Let now $\cal D$ be the difference between the left-hand side 
and the right-hand side of the relation to be proved, Equation~\eqref{mainadd}.
By the  process described above, $\cal D$ becomes a rational function,
with coefficients in $\Q( i\sqrt{3})$, 
in the \emph{four} quantities $X_1,Y_1,X_2,Y_2$, where $X_1=\wp(x/ i\sqrt{3})$,
$Y_1=\wp'(x/ i\sqrt{3})$, and similarly for $X_2,Y_2$, with~$y$ replacing~$x$. 
The (rather large) rational fraction 
normalizes to the form ${\cal D}=A/B$, with numerator~$A$ and denominator~$B$ involving, 
respectively, 2388 and 1256 monomials. One can then operate with
the rule (DEF), instantiated as
\[
Y_1^2\mapsto 4X_1^3+4, \qquad
Y_2^2\mapsto 4X_2^3+4,
\]
the corresponding reductions being simply effected by multivariate
polynomial divisions.
When this is done, we find that $A$ reduces to~0, while $B$ is reduced to a
\emph{nonzero} polynomial, which is of degree~1 in $Y_1,Y_2$ and of degree~8 in $X_1,X_2$.
The verification of the addition formula for $f(z)$ is thereby completed.
\end{proof}

A direct application of Theorem~\ref{sradd-thm} to the addition formula expressed by Proposition~\ref{add-prop} gives rise 
to our main continued fraction.

\begin{theorem}\label{main-thm} The ordinary generating function
 of the pseudo-factorials satisfies
\begin{equation}\label{confracaca}
F(z)
\equiv \sum_{n\ge0} \alpha_n z^n = \cfrac{1}{1+z+\cfrac{3\cdot 1^2\cdot z^2}{
1-z+\cfrac{2^2\cdot z^2}{1+3z+\cfrac{3\cdot 3^2\cdot z^2}{1-3z+\cfrac{4^2\cdot z^2}{\ddots}}}}},
\end{equation}
where the coefficients are, with the notations of the Jacobi form~\eqref{Jfrac}: 
\begin{equation}\label{caca}
c_j=(-1)^{j-1}\left(j+\frac{1+(-1)^j}{2}\right)
, \qquad
a_j=-j^2(2-(-1)^j).
\end{equation}
\end{theorem}
\begin{proof}
We make use of the addition formula~\eqref{mainadd},
which, taken under the form~\eqref{addf}, yields
\[
f(x+y)=\sum_{n=0}^\infty \omega_{2n}\varphi_{2n}(x)\varphi_{2n}(y)+
\sum_{n=0}^\infty \omega_{2n+1}\varphi_{2n+1}(x)\varphi_{2n+1}(y),
\]
with the $\omega_{n}$ and  $\varphi_n$ determined by
\[
\left\{\begin{array}{lll}
\omega_{2n}\varphi_{2n}(x)\varphi_{2n}(y)&=& + 3^{-n}f(x)f(y)\sigma(x)^n\sigma(y)^n\\
\omega_{2n+1}\varphi_{2n+1}(x)\varphi_{2n+1}(y)&=&
-3^{-n-1}h(x)h(y)\sigma(x)^n\sigma(y)^n\,.
\end{array}\right.
\]
We have, with the notations of~\eqref{addf}, $f(z)=1-z+O(z^2)$,
as well as $\sigma(z)\equiv 1-f(z)f(-z)=3z^2+O(z^4)$ and 
 $h(z)\equiv f(z)+f'(z)=-3z+O(z^3)$.
The required normalization of a Stieltjes--Rogers addition
formula, $\varphi_n(z)=z^n/n!+O(z^{n+1})$, combined with 
the low-order expansions of $f(z),h(z),\sigma(z)$, gives us
\begin{equation*}\label{omegas}
\omega_{2n}=3^n (2n)!^2, \qquad
\omega_{2n+1}=-3^{n+1}(2n+1)!^2,
\end{equation*}
as well as
\[
\varphi_{2n}(z)=\frac{z^{2n}}{(2n)!}-(2n+1)\frac{z^{2n+1}}{(2n+1)!}+O(z^{2n+2}),
\quad
\varphi_{2n+1}(z)=\frac{z^{2n+1}}{(2n+1)!}+O(z^{2n+3}).
\]
We thus have
$$\frac{\omega_1}{\omega_0}=-3,\quad\frac{\omega_2}{\omega_1}=-2^2,\
\frac{\omega_3}{\omega_2}=-3\cdot 3^2,\quad\frac{\omega_4}{\omega_3}=
-4^2,~\cdots,$$
and, for $c_j:=\varphi_{j,j+1}-\varphi_{j-1,j}$:
$$c_0=-1,\quad c_1=1,\quad c_2=-3,\quad c_3=3,\quad
c_4=-5,\quad c_5=5,
~\cdots\,.$$
By Theorem~\ref{sradd-thm}, the last two formulae conclude the proof of~\eqref{confracaca}.
\end{proof}

\section{\bf A family of orthogonal polynomials}\label{ortho-sec}

Our goal in this section consists in 
finding an explicit form for the polynomials that appear in the convergents
of the main continued fraction of Theorem~\ref{main-thm},
this by way of their exponential generating function.
We focus our attention  on the denominator polynomials, precisely, 
on their reciprocals, which form a family of formally
\emph{orthogonal polynomials} that appears to be new.

We start by specializing  to
the continued fraction under consideration~\eqref{confracaca}
some well-known algebraic properties found in~\cite{Perron57,Wall48}
that hold for an arbitrary Jacobi fraction~\eqref{Jfrac}.
The convergents of~\eqref{confracaca} are obtained by truncating the infinite fraction
before a numerator. In this way, a collection of rational fractions~$P_k(z)/Q_k(z)$
of increasing degrees is obtained,
\[
\frac{0}{1},\quad \frac{1}{1+z},
\quad{\frac {1-z}{1+2\,{z}^{2}}},
\quad {\frac {1+2\,
z+{z}^{2}}{1+3\,z+6\,{z}^{2}+10\,{z}^{3}}},
\quad {\frac {1-z+22\,{z}^{2}-30
\,{z}^{3}}{1+24\,{z}^{2}-8\,{z}^{3}+24\,{z}^{4}}},
\]
so that $Q_0=1$, $Q_1=1+z$, and so on.
The denominator  polynomials $Q_k$
satisfy a ``three-term recurrence'' relation,
\begin{equation}\label{Qrec}
Q_k=(1-c_{k-1}z)Q_{k-1}-a_{k-1}z^2Q_{k-2}.
\end{equation}
(The $P_k$ satisfy the same recurrence, but with initial
conditions $P_0=0$, $P_1=1$.) 
The reciprocal polynomials 
defined by
\begin{equation}\label{litq}
q_k(z)=z^k Q_k\left(\frac{1}{z}\right)
\end{equation}
then satisfy the recurrence
\begin{equation}\label{litqrec}
q_k=(z-c_{k-1})q_{k-1}-a_{k-1}q_{k-2},\qquad q_{-1}=0,\quad q_0=1,
\end{equation}
with the $a_j,c_j$ as in~\eqref{caca}. 
On general grounds, they are \emph{formally orthogonal} with respect to a (formal) measure
whose moments coincide with the pseudo-factorials,~$(\alpha_n)$. 
In other words, they are orthogonal with respect to the bilinear form
\[
\langle f,g\rangle=\langle fg\rangle , \qquad\hbox{with}\quad \langle z^n\rangle =\alpha_n.
\]
Observe finally that, once the $Q_k$ are known, the $P_k$ can somehow be regarded as known.
Indeed, relative to~\eqref{Jfrac}, one has,  in the sense of formal power series,
\[
Q_k(z)F(z)-P_k(z)=O(z^{2k}),
\]
so that the coefficients of the $P_k$ are 
expressible as a convolution of the two sequences
$[z^n]Q_k(z)$ and $\alpha_n\equiv [z^n]F(z)$.

We have the following characterization.

\begin{theorem} \label{qegf}
Let $\Upsilon(z,t)$ be the exponential generating function of 
the reciprocal polynomials $(q_k)$ of~\eqref{litq}, with coefficients~\eqref{caca}:
\[
\Upsilon(z,t):=\sum_{k=0}^\infty q_k(z) \frac{t^k}{k!}.
\]
Consider the algebraic curve
\begin{equation}\label{etadef}
2+3\,t+ 3t\left(1+t \right) \eta - 2\left(1-3\,{t}^{2}+3\,{t}^{4} \right)  \eta ^{3}=0,
\end{equation}
which is of genus~$0$ and is parametrized by
\begin{equation}\label{etaparam} 
 t=\frac{1}{3}{\frac { \left( {w}^{2}+3 \right) w}{{w}^{2}+1}},
\qquad
\eta=3\,{\frac { \left( w+1 \right)  \left( {w}^{2}+1 \right) }{{w}^{4}+3}
},
\end{equation}
and let $\eta(t)$ be the branch that satisfies $\eta(0)=1$:
\[
\eta(t)=1+t+2\frac{t^2}{2!}+10\frac{t^3}{3!}+24\frac{t^4}{4!}+
280\frac{t^5}{5!}+400\frac{t^6}{6!}+
12880\frac{t^7}{7!}-
\cdots\,.
\]
Define 
\begin{equation}\label{chidef}
\chi(t):=\sqrt{\eta(t)^2-\frac{2t(1+t)}{1-3t^2+3t^4}}
=1+\frac{t^2}{2!}-2\frac{t^3}{3!}+\frac{t^4}{4!}-100\frac{t^5}{5!}-575\frac{t^6}{6!}-\cdots
\end{equation}
and  introduce the fundamental elliptic integral
\begin{equation}\label{Jdef}
J(t):=\int_0^t \frac{du}{\sqrt{1-3u^2+3u^4}}
=t+3\frac{{t}^{3}}{3!}+45\frac{{t}^{5}}{5!}+1215\frac{{t}^{7}}{7!}+8505\frac{{t}^{9}}{9!}
-\cdots\,.
\end{equation}
Then, the generating function~$\Upsilon$ satisfies
\begin{equation}\label{great}
\Upsilon(z,t)=\eta(t)\cosh\left(z J(t)\right)+\chi(t)\sinh\left(z J(t)\right).
\end{equation}
\end{theorem}

Equation~\eqref{great}  was  first  arrived at  by   a combination  of
induction and of partly  heuristic calculations, based on ``guessing''
intermediate    differential equations as     well as on {\sc Maple}'s
symbolic integration capabilities.  Rather than offering a heavy proof
by  successive        transformations   of        the         defining
recurrence~\eqref{litqrec},    we  have   opted    to     present    a
computer-assisted verification of~\eqref{great}.  In this way, we feel
we  save symbols, hence pages, hence  trees. The price  to be paid was
only  a  few   hours  of interaction  with  the   symbolic  engine and
(eventually!) a few  seconds of computer  processing time.  As in  the
previous  section,  only  well-specified totally-algorithmic steps are
eventually  used.  Once   more, there  is  no  difficulty in  checking
intermediate steps against series   expansions  up to order  100   and
beyond.

Our proof of the identity~\eqref{great} eventually boils down to exhibiting 
a fourth-order differential operator in~$t$, with coefficients in~$\C(z)$, that is satisfied
by the difference between the two sides of~\eqref{great}. It is then sufficient to
check that both sides satisfy the same initial conditions
given by the coefficients of~$t^0$ up to~$t^3$.

The entire  process relies on the
\emph{holonomic framework} pioneered by Zeilberger~\cite{Zeilberger90},
with supporting theorems to be found in works of Stanley, such as~\cite{Stanley80} and~\cite[Ch.~6]{Stanley99}.
Let $\K$ be a ground field, which we take here to  be~$\C(z)$, the field of
rational fractions in~$z$.  (Throughout, we  treat the quantity~$z$ as
a  parameter.)  A formal power  series of~$\K[\![t]\!]$, simply called
``function'',    is   \emph{holonomic}   (alternative     names    are
differentiably   finite,  $D$--finite,  $\partial$--finite)  if it
satisfies  a linear differential   equation with  coefficients in  the
rational field $\K(t)$.  Equivalently, $h$  is holonomic if the vector
space over $\K(t)$ spanned  by all the derivatives  $\{\partial_t^j h\}$
is finite-dimensional.  Holonomic functions   are known to be   closed
under sum, product, differentiation, integration, and  algebraic
substitutions (i.e., substitutions of algebraic functions in place of variables). Finally, if~$h$ is holonomic, its sequence of coefficients $([z^n]h)$
satisfies a linear recurrence relation with coefficients in~$\K(n)$.

Clearly, a holonomic function is determined by
a  finite amount  of   information; namely,   a  defining differential
equation supplemented by sufficiently many initial conditions. 
Given two holonomic functions~$A,B$, one can then verify their 
conjectured identity as follows.
\begin{itemize}
\item[$(i)$] Compute a differential equation,  of order~$\omega$, say,
that is satisfied by the difference $A-B$.
\item[$(ii)$] Check the coincidence of 
the expansions of~$A$ and~$B$ up to terms of order $O(t^{\omega})$.
\end{itemize}
By the finiteness of the underlying vector spaces, the process constitutes a valid \emph{proof}
of~$A=B$, in ``non-singular'' cases at least\footnote{%
	An operator is said to be ``non-singular'' if the lead polynomial 
	of the associated recurrence (relating the coefficients~$h_n:=[t^n]h(t)$,
	of a solution~$h(t)$ and having coefficients in~$\K[n]$) has no root in~$\Z_{\ge0}$. 
	In the case of a non-singular operator of order~$\omega$, 
	the number of needed initial conditions equals~$\omega$.
	(In the ``singular'' case, 
	a higher, but still effectively computable, number may be needed.)
}. In our context,
it could be carried out
comparatively easily, thanks to the powerful {\tt Gfun} library developed by Salvy and Zimmermann~\cite{SaZi94}.

\def\Sh{\operatorname{{\bf S}}}\def\lo{\mathfrak{L}}\def\mo{\mathfrak{M}}

\begin{proof}[Proof (Theorem~\ref{qegf})]
In what follows, we use $\partial\equiv\partial_t$ to represent the differential
operator $\frac{\partial}{\partial t}$;
we denote by $\Sh\equiv \Sh_n$ the shift operator on infinite sequences $(u_n)$
such that  $\Sh(u_n)=u_{n+1}$. We let~$\Pi_r$ generically represent 
a polynomial of degree~$r$, either in~$t$ (for differential operators) or
in~$n$ (for difference operators),
with coefficients in~$\K$. As indicated before, the quantity~$z$ is
treated as a parameter. 
Our purpose is to prove $A=B$, where $A$ is the left-hand side of~\eqref{great} 
and~$B$ is the right-hand side:
\begin{equation}\label{AeqB}
A\mathop{=}^{?}B, \quad\hbox{with}~~  A:=\Upsilon(z,t), ~~ B:= \eta(t)
\cosh\left({z}J(t)\right)+\chi(t)\sinh\left({z}J(t)\right),
\end{equation}
and $\eta(t),\chi(t),J(t)$ as defined in the statement.
See Figure~\ref{AB-fig} for a summary of the main steps of our proof.


\smallskip
\emph{The left-hand side $(A)$.}
The parity inherent in the coefficients~\eqref{caca} suggests to introduce
the subsequences $r_n=q_{2n}$ and $s_n=q_{2n+1}$. The basic recurrence~\eqref{litqrec}
then relates $r_n$ to~$r_{n-1},s_{n-1}$ and $s_n$ to $r_n,s_{n-1}$; hence,
by substitution, the fact that the vector $(r_n,s_n)$ depends linearly on~$(r_{n-1},s_{n-1})$ 
via a matrix, whose coefficients are polynomial in~$n$ (and the parameter~$z$).
By instantiating this last relation at~$n+1$ and $n+2$, 
and using back substitution, there results  
that~$r_n$ and $s_n$ satisfy explicit linear recurrences of order~2 with
coefficients that are polynomial in~$n$.
The  difference operators annihilating~$(r_n)$ and~$(s_n)$ 
are found in this way to be of the form
\[
\Pi_1\Sh^2+\Pi_3\Sh^1+\Pi_5\Sh^0.
\]
Equivalently, the exponential generating functions 
\[
R(t)=\sum r_n \frac{t^{2n}}{(2n)!} \qquad\hbox{and}\qquad 
S(t)=\sum r_n \frac{t^{2n+1}}{(2n+1)!}
\]
are found to satisfy $\lo[R]=0$ and $\mo[S]=0$, where $\lo$ and~$\mo$ are each
of the form
\[
\Pi_5\partial^5+\Pi_4\partial^4+\cdots +\Pi_0\partial^0.
\]
In fact, second-order operators $\lo^{\circ}$ and
$\mo^{\circ}$ that appear to cancel~$R$ and~$S$, respectively, can be guessed
(roughly, by the method of indeterminate coefficients 
cleverly implemented in {\sc Maple}'s {\tt Gfun}).
The guesses can then be turned into full-fledged proofs by checking 
(with {\sc Maple}'s {\tt Ore\_algebra}, see~\cite{ChSa98}) the operator divisibility
relations: $\lo^{\circ}\,|\,\lo$ and~$\mo^{\circ}\,|\,\mo$.
Once this is done, a differential operator that annihilates $\Upsilon$ can be obtained by
making use of  properties of holonomic functions (effective closure
under sum) and  a further
round of simplification based on guessing. It is found in this way that
the second-order operator\footnote{
	Much to our surprise, {\sc Maple}'s symbolic integrator proposed 
	a solution to~$\frak{A}[f]=0$,
	which involved terms of
	the rough form $\exp\left(\pm {z} J(t)\right)$ and
	eventually led us to infer~\eqref{great}.
}
\begin{small}
\begin{equation}\label{Aop}\def\z{\zeta}\renewcommand{\arraycolsep}{1.5truept}
\begin{array}{lll}
\ds \frak{A}&=&4\left(1-3t^2+3t^4)((2t(t+1)\z-1\right)\partial^2
\\ && \ds {} +4
\left(24\,{t}^{5}\z+30\,{t}^{4}\z- \left(18+6\,\z \right) {t}^{3}-12\,{t}^{2}
\z- \left( 4\,\z-9 \right) t-2\,\z\right)
\partial^1
\\ && \ds {}
 +
\left(48\,{t}^{4}\z+72\,{t}^{3}\z-\left( 48+8\,{\z}^{3}-18\,\z \right) {t}^{2
} -\left( 6\,\z+8\,{\z}^{3}+12 \right) t+9+16\,\z+12\,{\z}^{2}\right)\partial^0,
\end{array}
\end{equation}
\end{small}%
with $\zeta:=z-1/2$, annihilates~$\Upsilon(z,t)$.

\begin{figure}\small
\[ \renewcommand{\arraycolsep}{2.5pt}
\begin{array}{lrlllll}
\hline
\ \\[-2mm]
A\equiv \Upsilon(z,t) \,:&  \Pi_6 \partial^2Y+\Pi_5\partial Y+\Pi_4 Y &=&0;&\hbox{see~$\frak{A}$ in~\eqref{Aop}}\\[2mm]
\eta(t), \chi(t) \,:&  \Pi_6 \partial^2Y+\Pi_5\partial Y+\Pi_4 Y &=&0;&\hbox{$\eta,\chi$ algebraic~\eqref{etadef}, \eqref{chidef}}\\[2mm]
\left\{\!\begin{array}{c}
\ds \exp\left(\pm {z}J(t)\right) \\[1mm]
\ds \cosh, \sinh\left({z}J(t)\right)
\end{array}\!\right\}\!: \quad&  \Pi_4 \partial^2 Y +\Pi_3\partial Y -z^2 Y &=&0;& 
\hbox{see~\eqref{2ndo}}\\[4mm]
\left\{\begin{array}{c}
\ds B^+\equiv \eta(t)\cosh\left({z}J(t)\right)
\\[1mm]
\ds B^-\equiv \chi(t)\sinh\left({z}J(t)\right)
\end{array}\right\}\!: &
\Pi_{12}\partial^4 Y 
+\cdots + \Pi_{8} Y &=&0;& 
\hbox{closure algorithm ($\times$)}\\[4mm]
\Delta:=A-(B^++B^-) & 
\Pi_{12}\partial^4 Y 
+\cdots + \Pi_{8} Y &=&0;& 
\hbox{closure algorithm ($+$)}\\[2mm]
\hline
\end{array}
\]
\caption{\label{AB-fig}\small
The shape of the differential equations satisfied by quantities intervening 
in the proof of the main equation~\eqref{great}.}
\end{figure}

\smallskip
\emph{The right-hand side $(B)$.} 
We can build up
differential equations starting with the explicit expression of~$B$ in~\eqref{AeqB}:
see again Figure~\ref{AB-fig} for a summary. 
Given a quantity $X$ that depends on~$z$, we set
$X^+=\frac12 X(z)+\frac12 X(-z)$ and $X^-=\frac12 X(z)-\frac12 X(-z)$,
defining its ``odd'' and ``even'' parts (in~$z$), respectively.
With $B\equiv\Upsilon(z,t)$, we then consider 
$B^+=\Upsilon^+=\eta(t)\cosh(zJ(t))$ and $B^-=\Upsilon^-=\chi(t)\sinh(zJ(t))$, and proceed to
construct the corresponding annihilators, $\frak{B}^+$ and~$\frak{B}^-$.

First, we observe that if $P(t)$ is an arbitrary polynomial,
then
\begin{equation}\label{2ndo}
Y(t):=\exp\left(z \int^t \frac{dw}{\sqrt{P(w)}}\right)
\quad\hbox{satisfies}\quad P \partial^2 Y +\frac12 P'\partial Y -z^2 Y=0.
\end{equation}
This equation is invariant by $z\leftrightarrow-z$, so that it is also satisfied 
when the exponential in~\eqref{2ndo} is replaced by~$\sinh,\cosh$. 

Next, the function~$\eta(t)$, given by a cubic algebraic equation,
is found to satisfy a second-order differential
equation with coefficients that are of degree at most~6.
The application of closure rules for products of holonomic functions then
provides for~$B^+\equiv\Upsilon^+$ a differential 
operator~$\frak{B^+}$ that is of order~4, with coefficients of degree at most~12.

We can then proceed to construct the annihilator~$\frak{B}^-$ 
of the odd part 
$B^-\equiv\Upsilon^-$. It turns out that the algebraic function 
$\chi(t)$ defined in~\eqref{chidef} 
satisfies \emph{the same differential equation} as $\eta(t)$, but with different
initial conditions ($\chi(0)=1$, $\chi'(0)=0$). There now results 
from this fact and the comments accompanying~\eqref{2ndo} that we can take
\[
\frak{B}=\frak{B^+}=\frak{B}^-,
\]
as annihilator of the right-hand side ($B$) of~\eqref{great}.

\smallskip
\emph{The comparison.}
Finally, it remains to verify that $A=B$. The operator $\Delta$ is defined
to annihilate ${\bf A}+{\bf B}$, where $\bf A$ and $\bf B$ are 
the vector spaces of solutions of $\frak{A}[f]=0$ and $\frak{B}[f]=0$,
respectively. By construction, the difference $A-B$ is such that $\Delta[A-B]=0$.
The operator, obtained by holonomic closure under sums,
is of type
\[
\Delta = \Pi_{12} \partial^4+\Pi_{11}\partial^3+\Pi_{10}
\partial^2+\Pi_9\partial^1+\Pi_8\partial^0.
\]
The associated recurrence operator is found to be of the form
$
\Pi_4 \Sh^{12}+\cdots + \Pi_4 \Sh^0,
$
with leading coefficient 
\[
(n+9)(n+10)(n+11)(n+12)(9+4z^2),
\]
so that $\Delta$ is non-singular at~$0$. It thus suffices to 
verify that the expansions of $A\equiv \Upsilon(z,t)$ and of the right-hand size $B$ 
in~\eqref{great} coincide till terms of order $O(t^4)$,
\[\hbox{$\ds
A,B=1+(z+1)t+(z^3+3z^2+6z+10)\frac{t^2}{2!}
+(z^4+24z^2-8z+24)\frac{t^3}{3!}+O(t^4),$}
\]
so as  to complete the proof that $A=B$.
Equation~\eqref{great} is now established.
\end{proof}

Orthogonal polynomials attached to continued fraction
expansions relative to elliptic functions, have been 
first studied by Carlitz and Al-Salam (see~\cite{IsMa99,IsVaYo01} for some more recent developments), 
and, as already mentioned, they form the subject of the monograph \emph{Elliptic Polynomials}
by Lomont and Brillhart ~\cite{LoBr01}. The family made explicit by Theorem~\ref{qegf}
does not appear to be captured by their classification and hence seems to be new.
Remarkably, in connection with birth-and-death processes having cubic weights,
Gilewicz \emph{et al.}~\cite{GiLeRuVa06} have recently discovered 
another new family of orthogonal polynomials, related to the expansion of the Dixonian function~sm
taken at~0 (as in~\cite{CoFl06}), rather than at the point $\pi_3/6$ 
that is needed here (cf Theorem~\ref{dix-thm}).

\section{\bf Consequences of the continued fraction expansion}\label{cong-sec}

The continued fraction of Theorem~\ref{main-thm} has several
interesting by-products that we now examine. These include
an explicit evaluation of Hankel determinants,
as well as elementary congruence properties of pseudo-factorials.

\paragraph{\bf\em Hankel determinants.}
It is well known that, generally, coefficients of a Jacobi fraction can be expressed
 as determinants. This fact is classically derived  from
Stieltjes's matrix version of the addition theorem~\cite[pp.~202--206]{Wall48};
it is equivalent to the $LDU$ decomposition of the Gram  matrix 
$H$, with entries $h_{i,j}=\langle z^i,z^j\rangle\equiv\langle z^{i+j}\rangle$ (for $i,j\ge0$),
which is also known as the Hankel matrix of the sequence~$\langle z^n\rangle$;
see for instance~\cite[\S2.1]{Ismail05}.
Conversely, any known continued fraction yields an explicit Hankel determinant evaluation.
Given this, an immediate consequence of Theorem~\ref{main-thm} is
the following.

\begin{corollary} Let~$m$ be a positive integer. 
The Hankel determinant of pseudo-factorials 
\[
H_m^{(0)}:=
\left| \begin{array}{cccc}
\alpha_0 & \alpha_1& \cdots & \alpha_{m-1} \\
\alpha_1 & \alpha_2& \cdots & \alpha_{m} \\
\vdots & \vdots & \ddots & \vdots \\
\alpha_{m-1} & \alpha_m& \cdots & \alpha_{2m-2} 
\end{array}
\right| 
\]
admits the closed form
\[
H_m^{(0)}= \prod_{j=1}^{m-1} a_j^{m-j}=
\left\{ \begin{array}{ll}
\ds (-1)^{m/2}3^{m^2/4}\left(\prod_{k=1}^{m-1} k!\right)^2 & \hbox{($m$ even)}\\
\ds (-1)^{(m-1)/2}3^{(m^2-1)/4}\left(\prod_{k=1}^{m-1} k!\right)^2 & \hbox{($m$ odd)},
\end{array}\right.
\]
where the~$a_j=-j^2(2-(-1)^j)$ are the continued fraction numerators of~\eqref{caca}.
\end{corollary}

\paragraph{\bf\em Congruences.}

\def\MOD{\operatorname{mod}}

\begin{figure}\small
\begin{center}
\begin{small}
\renewcommand{\tabcolsep}{1.2pt}
\begin{tabular}{r|rrrrrrrrrrrrrrrrrrrrrrrrrrr}
\hline
\hline
$M$ &$n={}$& 0&1&2&~3&4&5&6&7&8&9&10&11&12&13&14&15&16&17&18&19&20&21&22&23&24&25\\
\hline
2&&1&1&0&0&0&0&0&0&0&0&0&0&0&0&0&0&0&0&0&0&0&0&0&0&0&0\\
3&&1&2&1&2&1&2&1&2&1&2&1&2&1&2&1&2&1&2&1&2&1&2&1&2&1&2\\
4&&1&3&2&2&0&0&0&0&0&0&0&0&0&0&0&0&0&0&0&0&0&0&0&0&0&0\\
5&&1&4&3&2&1&0&0&0&0&0&0&0&0&0&0&0&0&0&0&0&0&0&0&0&0&0\\
6&&1&5&4&2&4&2&4&2&4&2&4&2&4&2&4&2&4&2&4&2&4&2&4&2&4&2\\
7&&1&6&5&2&2&2&2&4&1&6&6&6&6&5&3&4&4&4&4&1&2&5&5&5&5&3\\
8&&1&7&6&2&0&0&0&0&0&0&0&0&0&0&0&0&0&0&0&0&0&0&0&0&0&0\\
9&&1&8&7&2&7&5&4&5&1&8&1&2&7&2&4&5&4&8&1&8&7&2&7&5&4&5\\
10&&1&9&8&2&6&0&0&0&0&0&0&0&0&0&0&0&0&0&0&0&0&0&0&0&0&0\\
11&&1&10&9&2&5&4&10&6&5&1&1&0&0&0&0&0&0&0&0&0&0&0&0&0&0&0\\
12&&1&11&10&2&4&8&4&8&4&8&4&8&4&8&4&8&4&8&4&8&4&8&4&8&4&8\\
13&&1&12&11&2&3&12&5&0&5&1&4&2&1&11&9&4&6&11&10&0&10&2&8&4&2&9\\
14&&1&13&12&2&2&2&2&4&8&6&6&6&6&12&10&4&4&4&4&8&2&12&12&12&12&10
\\
15&&1&14&13&2&1&5&10&5&10&5&10&5&10&5&10&5&10&5&10&5&10&5&10&5&
10&5\\
16&&1&15&14&2&0&8&0&0&0&0&0&0&0&0&0&0&0&0&0&0&0&0&0&0&0&0\\
17&&1&16&15&2&16&11&3&3&5&16&7&12&10&7&1&10&1&0&0&0&0&0&0&0&0&0\\
18&&1&17&16&2&16&14&4&14&10&8&10&2&16&2&4&14&4&8&10&8&16&2&16&14
&4&14\\
19&&1&18&17&2&16&17&3&14&0&17&6&9&8&3&18&0&15&8&7&11&3&16&14&3&5
&17\\
20&&1&19&18&2&16&0&0&0&0&0&0&0&0&0&0&0&0&0&0&0&0&0&0&0&0&0\\
\hline\hline\end{tabular}
\end{small}
\end{center}
\caption{\label{cong-fig}\small
The congruences $(\alpha_n \bmod M)$ for $M=2,\ldots,20$
and $n=0,\ldots,25$.
}
\end{figure}

A cursory examination of the $\alpha_n$ suggests clear divisibility patterns;
for instance, from~\eqref{inivals}, we immediately expect the $\alpha_n$ to be
divisible by~10, for $n$ large enough.
Figure~\ref{cong-fig} tabulates arithmetic congruence properties of the $\alpha_n$ for
small values of the modulus~$M$ and of the index~$n$.
The table obviously has much structure: the sequence~$(\alpha_n)$
appears to be eventually~$0$ modulo the numbers $2,4,5,8,10,11,16,17,20$;
there are obvious periodically reproducing
patterns, such as  $\overline{1,2}$ (mod~3), $\overline{4,2}$ (mod~6),
or the more recondite, and curiously repetitive,
\[
\hbox{$\overline{6,5,2,2,2,2,4,1,6,6,6,6,5,3,4,4,4,4,1,2,5,5,5,5,3,6,1,1,1,1,2,4,3,3,3,3}$}
\]
of length~36 corresponding to modulus~7. We state here a simple consequence of
our main continued fraction~\eqref{confracaca} in Theorem~\ref{main-thm}.

\begin{corollary} \label{cong-thm}
$(i)$~The sequence~$(\alpha_n)$ of pseudo-factorials 
is eventually periodic modulo any integer $M\ge2$.
$(ii)$~For each~$m\ge2$, the sequence~$(\alpha_n)$
satisfies modulo~$M=3^{\lceil m/2\rceil}m!^2$ a 
linear recurrence with constant coefficients that is of order at most~$m$.
\end{corollary}
\begin{proof}[Proof (sketch)]
The statement is an instance of the general fact that $J$--fraction expansions with
integer coefficients automatically imply congruence properties~\cite{Flajolet82}.

$(i)$~The main continued fraction~\eqref{confracaca} representing
the ordinary generating function~$F(z)$ of pseudo-factorials has a factor 
of~$M^2$ at its numerator of rank $(M+1)$. 
In particular, the contributions induced by the stages
$(M+1)$, $(M+2)$, and so on, of this continued fraction are zero modulo~$M^2$. 
In other words, $F(z)$
is congruent modulo~$M^2$ to the $M$th convergent of the continued fraction. 
Thus, it satisfies, modulo~$M^2$, a linear recurrence of order at most~$M$.
Hence it is eventually periodic modulo~$M$.

$(ii)$~The estimate above of the order of the recurrence satisfied by modular
reductions of the pseudo-factorials can be vastly improved~\cite{Flajolet82}.
By the classical $(2\times2)$--determinant 
identity of orthogonal polynomials and 
 convergent denominators, the difference of two successive convergents of~\eqref{Jfrac}
satisfies the identity
\[
\frac{P_{k+1}(z)}{Q_{k+1}(z)}-\frac{P_{k}(z)}{Q_{k}(z)}=
\frac{a_1 a_2\cdots a_{k}z^{2k}}{Q_{k}(z)Q_{k+1}(z)}.
\]
This specializes to the $J$--fraction~\eqref{confracaca} relative to~$F(z)$,
when the $a_j$ are taken to be as in~\eqref{caca}.
By expressing that $F(z)$ is the sum of the differences of its 
successive convergents, we then
obtain, for any~$m\ge0$,
\[
F(z)=\frac{P_m(z)}{Q_m(z)}+\sum_{k\ge m}
\frac{a_1 a_2\cdots a_{k}z^{2k}}{Q_{k}(z)Q_{k+1}(z)}.
\]
In particular, since the~$a_j$ are all integers and the $Q_j$ are integral with $Q_j(0)=1$,
we have,
\[
F(z) \equiv \frac{P_m(z)}{Q_m(z)} \quad (\MOD M), \qquad\hbox{with $M=a_1a_2\cdots a_{m}$},
\]
in the sense that coefficients of both series are equal, after reduction modulo~$M$.
Thus, modulo~$M$, the $\alpha_n$ satisfy a linear recurrence whose characteristic 
polynomial is exactly the denominator polynomial~$Q_m(z)$ (reduced mod~$M$).
\end{proof}

As an illustration, corresponding to~$m=1,2,3$, we find the congruences
\[
\begin{array}{l}
\ds F(z)\equiv \frac{1}{1+z} \quad (\MOD~3\cdot 1!^2),
\qquad
F(z)\equiv \frac{1-z}{1+2z^2} \quad(\MOD~ 3\cdot 2!^2),
\\
\ds F(z)\equiv \frac{1+2z+z^2}{1+3z+6z^2+10z^3}
\quad (\MOD~3^2\cdot 3!^2),
\end{array}
\]
which already justify the data of Figure~\ref{cong-fig}
for moduli $2,3,4,6,9,12,18$; for instance,
from the convergent $P_1/Q_1$, we find $\alpha_n \equiv (-1)^n ~~(\MOD 3)$.
For $m=7$, the form
\[
\frac{P_7(z)}{Q_7(z)}
\equiv 5+\frac{3+6z+5z^2+2z^3+2z^4+2z^5}{1+4z^6} \quad (\MOD~7)
\]
explains the observed patterns of~$(\alpha_n)$ modulo~7 and the period equal to~36. 
By contrast, for $m=11$, we find that
\[
\frac{P_{11}(z)}{Q_{11}(z)}
\equiv 1+10z+9z^2+2z^3+5z^4+4z^5+10z^6+6z^7+5z^8+z^9+z^{10} \quad (\MOD~11)
\]
($Q_{11}$ reduces to~1 modulo 11),
thereby establishing that the $\alpha_n$ with $n\ge11$ are all divisible by~11.

As the previous discussion suggests, congruence properties 
of pseudo-factorials are tightly linked to arithmetic properties
of the $Q$ polynomials whose exponential generating function
has been determined in~Theorem~\ref{qegf}. 
Let again~$\Pi_r$ denote an unspecified polynomial of degree~$r$.
Without attempting a general discussion,
we only remark here the existence of striking regularities,
as summarized by the following data. First, for $m$ a prime of the form $6\mu+1$:
\[
\hbox{\renewcommand{\arraystretch}{1.9}$\ds
\begin{array}{cccc}
\hline\hline
\\[-10mm] \MOD~7 & \MOD~ 13 & \MOD~ 19 & \MOD~31\\
\hline
\ds \frac{P_7}{Q_7}\equiv \frac{\Pi_6}{1+4z^6} &
\ds \frac{P_{13}}{Q_{13}}\equiv \frac{\Pi_{12}}{1+11z^{12}} &
\ds \frac{P_{19}}{Q_{19}}\equiv \frac{\Pi_{18}}{1+11z^{18}} &
\ds \frac{P_{31}}{Q_{31}}\equiv \frac{\Pi_{30}}{1+4z^{30}}.\\[2.5mm]
\hline\hline\end{array}
$}
\]
Finally, for $m$ a prime of the form $6\mu+5$:
\[
\hbox{\renewcommand{\arraystretch}{1.9}$\ds
\begin{array}{cccc}
\hline\hline
\\[-10mm] \MOD~5 & \MOD~ 11 & \MOD~ 17 & \MOD~23\\
\hline
\ds \frac{P_5}{Q_5}\equiv {\Pi_4} &
\ds \frac{P_{11}}{Q_{11}}\equiv {\Pi_{10}} &
\ds \frac{P_{17}}{Q_{17}}\equiv {\Pi_{16}} &
\ds \frac{P_{23}}{Q_{23}}\equiv {\Pi_{22}}.\\[2.5mm]
\hline\hline\end{array}
$}
\]


\section{Conclusion}\label{concl-sec}

The relation between elliptic functions and  continued fractions is an
old  subject, one that is especially  rich.  Connections  are manifest with
the theta    function   framework  in  the   form  of  various   types
of~$q$--series expansions,  starting with Eisenstein and including the
celebrated      Rogers--Ramanujan  identities~\cite[Ch.~16]{Berndt91}.
Closer to our perspective are early contributions due to Stieltjes and
Rogers regarding  the Jacobian sn,cn  framework.  Conrad, first in his
dissertation~\cite{Conrad02}      then       in   collaboration    with
Flajolet~\cite{CoFl06}, has elicited new connections with the Dixonian
framework of the sm,cm functions.  As the present work supplemented by
further investigations  of  ours indicate, there are   new
continued fractions to be explored, attached to the Weierstra{\ss} and
Dixonian frameworks.  In this vein, we have recently discovered new 
elliptic continued fractions,
relative to ``equiharmonic'' and ``lemniscatic'' numbers, which are
lattice analogues of the pseudo-factorials---we plan to report on these 
in a future publication.

\bigskip
\noindent
\begin{small}%
{\bf Acknowledgements.}  The work of P. Flajolet was supported in part by
the  SADA and LAREDA Projects  of  the French National Research Agency
(ANR).  The   authors are  grateful  to  Fr\'ed\'eric   Chyzak, Manuel
Kauers,  and Bruno Salvy  for  insightful discussions  relative to 
computer verification of identities.   They are also indebted to Bruno
Salvy for making  available his upgraded version  of the {\sc Maple} {\sf
Gfun} package on   which the developments  of  Section~\ref{ortho-sec}
could be most effectively built. Thanks finally to Robin Chapman, an anonymous referee,
and the editor for encouraging remarks and constructive suggestions. \par
\end{small}


\def\cprime{$'$}

 \end{document}